\documentclass[10pt, a4paper, reqno]{amsart}
\pdfoutput=1
\usepackage{amsmath,amsfonts,amssymb,amsthm,color}  
\usepackage{mathrsfs}

\usepackage[utf8]{inputenc}
\usepackage{graphicx}

\usepackage{marginnote}

\theoremstyle{plain}
\newtheorem{theorem}{Theorem}[section]

\newtheorem{lemma}[theorem]{Lemma}
\newtheorem{proposition}[theorem]{Proposition}

\theoremstyle{definition}

\newtheorem{remark}[theorem]{Remark}

\numberwithin{equation}{section}
\newtheorem*{theorem*}{Theorem}

\newcommand{\R}{{\mathbb R}}
\newcommand{\N}{{\mathbb N}}

\newcommand{\CM}{{\mathcal{M}}}
\newcommand{\CB}{{\mathcal{B}}}
\newcommand{\CR}{{\mathcal{R}}}

\def\vint_#1{\mathchoice%
          {\mathop{\kern 0.2em\vrule width 0.6em height 0.69678ex depth -0.58065ex
                  \kern -0.8em \intop}\nolimits_{\kern -0.4em#1}}%
          {\mathop{\kern 0.1em\vrule width 0.5em height 0.69678ex depth -0.60387ex
                  \kern -0.6em \intop}\nolimits_{#1}}%
          {\mathop{\kern 0.1em\vrule width 0.5em height 0.69678ex
              depth -0.60387ex
                  \kern -0.6em \intop}\nolimits_{#1}}%
          {\mathop{\kern 0.1em\vrule width 0.5em height 0.69678ex depth -0.60387ex
                  \kern -0.6em \intop}\nolimits_{#1}}}
\def\vintslides_#1{\mathchoice%
          {\mathop{\kern 0.1em\vrule width 0.5em height 0.697ex depth -0.581ex
                  \kern -0.6em \intop}\nolimits_{\kern -0.4em#1}}%
          {\mathop{\kern 0.1em\vrule width 0.3em height 0.697ex depth -0.604ex
                  \kern -0.4em \intop}\nolimits_{#1}}%
          {\mathop{\kern 0.1em\vrule width 0.3em height 0.697ex depth -0.604ex
                  \kern -0.4em \intop}\nolimits_{#1}}%
          {\mathop{\kern 0.1em\vrule width 0.3em height 0.697ex depth -0.604ex
                  \kern -0.4em \intop}\nolimits_{#1}}}

\newcommand{\intav}{\vint}
\newcommand{\aveint}[2]{\mathchoice%
          {\mathop{\kern 0.2em\vrule width 0.6em height 0.69678ex depth -0.58065ex
                  \kern -0.8em \intop}\nolimits_{\kern -0.45em#1}^{#2}}%
          {\mathop{\kern 0.1em\vrule width 0.5em height 0.69678ex depth -0.60387ex
                  \kern -0.6em \intop}\nolimits_{#1}^{#2}}%
          {\mathop{\kern 0.1em\vrule width 0.5em height 0.69678ex depth -0.60387ex
                  \kern -0.6em \intop}\nolimits_{#1}^{#2}}%
          {\mathop{\kern 0.1em\vrule width 0.5em height 0.69678ex depth -0.60387ex
                  \kern -0.6em \intop}\nolimits_{#1}^{#2}}}

\def\XXint#1#2#3{{\setbox0=\hbox{$#1{#2#3}{\int}$}
\vcenter{\hbox{$#2#3$}}\kern-.5\wd0}}

\newcommand{\dy}{\, \mathrm{d}y}
\newcommand{\dx}{\, \mathrm{d}x}
\newcommand{\dt}{\, \mathrm{d}t}

\title[Continuity fractional maximal function]{Endpoint Sobolev continuity of the fractional maximal function in higher dimensions}

\thanks{The first author was supported by the ERCEA Advanced Grant 2014 669689 - HADE, by the MINECO project MTM2014-53850-P, by Basque Government project IT-641-13 and also by the Basque Government through the BERC 2018-2021 program and by Spanish Ministry of Economy and Competitiveness MINECO: BCAM Severo Ochoa excellence accreditation SEV-2017-0718. The second author acknowledges the UCLA Department of Mathematics for support.}

\author{David Beltran and José Madrid}
\date{\today}

\address{David Beltran, Basque Center for Applied Mathematics (BCAM), Alameda de Mazarredo 14, 48009, Bilbao, Basque Country, Spain}
	\email{dbeltran@bcamath.org}
	
\address{José Madrid, Department of  Mathematics,  University  of  California,  Los  Angeles (UCLA),  Los  Angeles,
California, 90024, USA}
\email{jmadrid@math.ucla.edu}


\begin{document}

\begin{abstract}
We establish continuity mapping properties of the non-centered fractional maximal operator $M_{\beta}$ in the endpoint input space $W^{1,1}(\R^d)$ for $d \geq 2$ in the cases for which its boundedness is known. More precisely, we prove that for $q=d/(d-\beta)$ the map $f \mapsto |\nabla M_\beta f|$ is continuous from $W^{1,1}(\R^d)$ to $L^{q}(\R^d)$ for $ 0 < \beta < 1$ if $f$ is radial and for $1 \leq \beta < d$ for general $f$. The results for $1\leq \beta < d$  extend to the centered counterpart $M_\beta^c$. Moreover, if $d=1$, we show that the conjectured boundedness of that map for $M_\beta^c$ implies its continuity.
\end{abstract}

\maketitle




\section{Introduction}
Given $f\in L^1_{loc}(\R^d)\,$ and $0\leq \beta<d\,$, the non-centered fractional Hardy-Littlewood maximal operator $M_{\beta}$ is defined by
\begin{equation*}
M_{\beta} f(x):=\sup_{\bar B(z,r)\ni x}
\frac{r^\beta}{|B(z,r)|}\int_{B(z,r)}|f(y)|\dy
\end{equation*}
for every $x\in\R^d\,$. The centered version of $M_{\beta}$, denoted by $M^{c}_\beta$, is defined by taking the supremum over all balls centered at $x$. The non-fractional case $\beta=0$ corresponds to the classical maximal function, which we denote by $M=M_0$.

In recent years, there has been considerable interest in understanding the regularity properties of $M$ and $M_\beta$. This study was initiated by Kinnunen \cite{Kinnunen1997}, who showed that if $f \in W^{1,p}(\R^d)$ with $1 < p <\infty$, then $Mf \in W^{1,p}(\R^d)$ and
\begin{equation}\label{eq:Kinnunen}
|\nabla M f(x)| \leq M(|\nabla f|)(x)
\end{equation}
almost everywhere in $\R^d$. His result extends in a straightforward way to the fractional case in the scaling line $\frac{1}{q}=\frac{1}{p} - \frac{\beta}{d}$; more generally, any $L^p-L^q$ bounded sublinear operator $A$ on $\R^d$ that commutes with translations preserves the boundedness at the derivative level if $1<p,q< \infty$, that is
$$
\| A f \|_{1,q} \leq C \|f \|_{1,p}.
$$
At the endpoint $p=1$, one cannot expect boundedness of $M_\beta$ from $W^{1,1}$ to $W^{1,\frac{d}{d-\beta}}$ to hold, as $M_\beta$ fails to be bounded at the level of Lebesgue spaces. However, one may still ask the question of whether the map $f \mapsto |\nabla M_\beta f|$ is bounded from $W^{1,1}$ to $L^{\frac{d}{d-\beta}}$. This problem has received a lot of attention in recent years and in the case $\beta=0$ is commonly referred to as the $W^{1,1}$--problem. In this case, despite the question is still open, there are positive results for $d=1$ \cite{Tanaka2002,AP2007,Kurka2010} and for $d>1$ if the function $f$ is radial \cite{Luiro2017}; see also \cite{HM2010, CS2013, CFS2015, Saari2016, Ramos2017, PPSS2017, Korry} for related results and \cite{BCHP2012, CH2012, Madrid2016} for similar results in the discrete setting. In the fractional case $0<\beta<d$, it was observed by Carneiro and the second author \cite{CM2015} that the case $\beta \geq 1$ follows from combining Sobolev embeddings with the following smoothing property of fractional maximal functions due to Kinnunen and Saksman \cite{KS2003}: if $f \in L^p(\R^d)$ with $1<p<d$ and $1 \leq \beta < d/p$, then
\begin{equation}\label{KS}
|\nabla M_\beta f(x)| \leq C M_{\beta-1} f(x)
\end{equation}
almost everywhere in $\R^d$. Together with the boundedness of $M_{\beta-1}$ and the Gagliardo--Nirenberg--Sobolev inequality,
$$
\| \nabla M_\beta f \|_{q} \leq C \|M_{\beta-1} f \|_q \leq C \| f \|_{\frac{d}{d-1}} \leq C \| \nabla f \|_1
$$
for $q=\frac{d}{d-\beta}$, establishing the endpoint Sobolev bound for $\beta \geq 1$. Here and in \eqref{KS} the results continue to hold for $M_\beta^c$.

The case $0<\beta < 1$ is considerably more difficult. The one dimensional case was established by Carneiro and the second author \cite{CM2015}, whilst in higher dimensions Luiro and the second author \cite{LM2017} proved its validity for radial functions. More recently, the first author, Ramos and Saari \cite{BRS2018} obtained the boundedness result for $d \geq 2$ without the radial hypothesis but for certain variants of $M_\beta$. Such variants correspond to a lacunary version of the maximal function $M_\beta$ and to maximal functions of convolution type with smoother kernels than $\chi_{B(0,1)}$. 

The maximal functions $M_\beta$ are sublinear operators, and therefore its boundedness on Lebesgue spaces implies its continuity. However, this property is not preserved at the derivative level: the map $f \mapsto |\nabla M_\beta f|$ is no longer sublinear. Therefore, it is a non-trivial question to determine the continuity of $f \mapsto |\nabla M_\beta f|$ as a map from $W^{1,p}(\R^d)$ to $L^{q}(\R^d)$. This question was first posed by Haj\l asz and Onninen \cite{HO2004}, where it was attributed to Iwaniec. The first affirmative results in this direction were obtained by Luiro \cite{Luiro2007} for $\beta=0$ in the non-endpoint cases $p>1$, although his analysis extends to the fractional setting; see also his work \cite{Luiro2018} for more general maximal operators in non-endpoint cases, which includes an interesting result for $M_\beta$ in the case $1 \leq \beta < d$.

In analogy to the boundedness problem, the continuity at the endpoint $p=1$ is a much subtler question. In recent years, there has been progress in this direction for $d=1$: Carneiro, the second author and Pierce \cite{CMP2017} established the continuity for $d=1$ and $\beta=0$, and the second author \cite{Madrid2017} showed the analogous result for $d=1$ and $0< \beta <1$. The main goal of this paper is to explore the analogous questions in higher dimensions for the cases in which the boundedness of the map $f \mapsto |\nabla M_\beta f|$ from $W^{1,1}$ to $L^q$ is known. In particular, we obtain positive results for the fractional case. Similarly to the boundedness, our analysis naturally splits in two cases depending on whether $0 < \beta < 1$ or $1 \leq \beta < d$; this is dictated by the availability of \eqref{KS} in the latter case.

\begin{theorem}\label{thm:beta bigger than 1}
Let $\mathcal{M}_\beta \in \{M_\beta, M_\beta^c\}$. If $1 \leq \beta < d$, the operator $f \mapsto |\nabla \mathcal{M}_\beta f|$ maps continuously $W^{1,1}(\R^d)$ into $L^{d/(d-\beta)}(\R^d)$.
\end{theorem}

The range $0 < \beta < 1$ is more interesting as the inequality \eqref{KS} is no longer at our disposal. However, we are able to give positive results for radial functions; note that boundedness of $M_\beta$ at its derivative level is currently only known under this assumption. This constitutes the main result of this paper.\footnote{The space $W^{1,1}_{\mathrm{rad}}$ in Theorem \ref{thm:beta smaller than 1} denotes the subspace of $W^{1,1}$ consisting of radial functions.}

\begin{theorem}\label{thm:beta smaller than 1}
If $0 <\beta < 1$, the operator $f \mapsto |\nabla M_\beta f|$ maps continuously $W^{1,1}_{\mathrm{rad}}(\R^d)$ into $L^{d/(d-\beta)}(\R^d)$.
\end{theorem}

The proof of this theorem differs from its one dimensional counterpart, which strongly uses that $M f$ and $M_\beta f$ are in $L^\infty(\R)$ if $f \in W^{1,1}(\R)$. In fact, the one-dimensional arguments will only continue to work in higher dimensions in the restricted range $d-1 < \beta < d$ which is, in particular, covered by Theorem \ref{thm:beta bigger than 1}, or inside an annulus $A(a,b)$ with $0<a<b<\infty$ in the case we consider small radii. Our approach will combine the one-dimensional arguments in $A(a,b)$ and a refinement of the techniques used in \cite{LM2017} to show the bound $\| \nabla M_\beta f \|_{d/(d-\beta)} \leq C \| \nabla f \|_{1}$ for radial functions.

Moreover, our arguments can be combined with those in \cite{Madrid2017} to yield a conjectural result in one dimension regarding the continuity of the map $f \mapsto |(M_\beta^c f)'|$ from $W^{1,1}(\R)$ to $L^{1/(1-\beta)}(\R)$. Our result depends upon the boundedness of that map between such function spaces, which is currently an open question.

\begin{theorem}\label{thm:d=1}
Let $0 < \beta < 1$. Assume that $\| (M^{c}_\beta f)' \|_{L^{q}(\R)} \leq C \|f' \|_{L^1(\R)}$ holds for $q=1/(1-\beta)$. Then the operator $f \mapsto |(M_\beta^c f)'|$ maps continuously $W^{1,1}(\R)$ into $L^q(\R)$.
\end{theorem}

Finally, it is noted that some of our arguments also continue to work without the radial assumption, for $\beta=0$ and for the centered maximal function. In particular, the analysis can always be reduced to showing the continuity inside a compact set $K$; this will be discussed in Section \ref{subsec:outside compact}.
\\

\textit{Structure of the paper.} Section \ref{sec:preliminaries} contains many auxiliary results that will be used in the proofs of the main theorems. The proofs of Theorems \ref{thm:beta bigger than 1},  \ref{thm:beta smaller than 1} and \ref{thm:d=1} are provided in Sections \ref{sec:beta big}, \ref{sec:beta small} and \ref{sec:proof thm d=1} respectively. Finally, an alternative proof for the range $\beta \in (d-1 , d)$ based on a one dimensional analysis will be provided in an Appendix.
\\

\textbf{Acknowledgements.} The authors are indebted to Hannes Luiro for a clarification regarding his previous work \cite{Luiro2007} and to Emanuel Carneiro, Cristian González and Juha Kinnunen for valuable comments. They also would like to thank BCAM, ICTP and UCLA for supporting research visits that helped to the development of this project. The second author would like to thank Carlos Pérez for his hospitality during his visit to BCAM. The authors are also grateful to the referee for the valuable suggestions.




\section{Preliminaries}\label{sec:preliminaries}




\subsection*{Notation}

Throghout this paper, the value of the Lebesgue exponent $q$ will always be $q=d/(d-\beta)$. Given a measurable set $E \subset \R^d$, $\chi_E$ denotes the characteristic function of $E$ and $E^c:=\R^d \backslash E$  its complementary set in $\R^d$. For $c \in \R$ , we denote by $cE$ the concentric set to $E$ dilated by $c$. The integral average of $f \in L^1_{loc}(\R^d)$ over $E$ is denoted by $f_E=\intav_E f$. The notation $A \lesssim B$ is used if there exists $C>0$ such that $A \leq C B$, and similarly $A \gtrsim B$ and $A \sim B$. The implicit constant may change from line to line but will be always independent of the relevant parameters (such as the index $j$), and depend only on the dimension $d$ and the fractional order $\beta$. The volume of the $d$-dimensional unit ball is denoted by $\omega_d$.




\subsection{The families of good balls and good radii}

Fix $0 \leq \beta < d$. Given a function $f \in W^{1,1}(\R^d)$ and a point $x \in \R^d$, define the family of \textit{good balls} for $f$ at $x$ as
$$
\CB_x^\beta (f):=\Big\{ B(z,r): x \in \bar B(z,r), \:M_\beta f (x)= r^\beta \intav_{B(z,r)} |f(y)| \dy \Big\}.
$$
Note that $\CB_x^\beta (f) \neq \emptyset$ for all $x \in \R^d$ if $f \in L^1(\R^d)$.  Moreover, $\CB_x^\beta(f)$ is a compact set in the sense that if $B(z_k,r_k) \in \CB_x^\beta(f)$ for all $k \in \N$ and $z_k \to z$ and $r_k \to r$ as $k \to \infty$, then $B(z,r) \in \CB_x^\beta(f)$.

For ease of notation, $\CB_x^\beta (f)$ will be simply denoted by $\CB_x^\beta$, and given a sequence of functions $\{f_j\}_{j \in \N}$ the associated families of good balls  are denoted by $\CB_{x,j}^\beta$. The families of \textit{good radii} $\CR_{x}^\beta$ and $\CR_{x,j}^\beta$ are defined as the subsets of $\R$ consisting of the radii associated to good balls in $\CB_x^\beta$ and $\CB_{x,j}^\beta$ respectively.

If $0< \beta < d$, the value $r=0 \not \in \CR_{x}^\beta$ for almost every $x \in \R^d$. This is indeed a simple consequence of the Lebesgue differentiation theorem. Assume that $\{B(z_k,r_k)\}_{k \in \N}$ is a family of balls containing $x$ such that $r_k \to 0$; then $\{x\}=B(x,0) \in \CB_x^\beta $ by compactness. By Lebesgue differentiation theorem
$$
r_k^\beta \intav_{B(z_k,r_k)} |f| \to 0 \times f(x) = 0 \quad \textrm{ a.e. \:\: as $k \to \infty$},
$$
but $M_\beta f (x)>0$ for any $f$ not identically zero.

If $\beta=0$, a similar argument yields that $0 \not \in \CR^0_x$ on the set $\{x \in \R^d : Mf(x)>f(x)\}$.

An important observation is the following relation between the sets $\CB_{x.j}^\beta$ and $\CB_x^\beta$, which constitutes the fractional higher dimensional analogue of Lemma 12 in \cite{CMP2017}.

\begin{lemma}\label{lemma:limit of good balls}
Let $f \in W^{1,1}(\R^d)$ and $\{f_j\}_{j \in \N} \subset W^{1,1}(\R^d)$ be such that $\| f_j - f \|_{W^{1,1}(\R^d)} \to 0$ as $j \to \infty$. For a.e. $x \in \R^d$, let $\{(z_j, r_j)\}_{j \in \N} \subset \R^d \times [0,\infty)$ be a sequence of centers and radii such that  $B_{x,j}=B(z_j,r_j) \in \CB_{x,j}^\beta$. If $(z, r)$ is an accumulation point of $\{(z_j, r_j)\}_{j \in \N}$, then $B(z,r) \in \CB_x^\beta $.
\end{lemma}

\begin{proof}
Set $f_0=f$, and for every $j \geq 0$ let $E_j$ be the set of the Lebesgue points of $f_j$. Define $E=\cap_{j\geq0} E_{j}$; note $\R^{d}\setminus{E}$ is a set of measure zero. Consider a point $x\in E$ and assume, without loss of generality, that $(z_j, r_j) \to (z,r)$ as $j \to \infty$ (going through a subsequence, if necessary) and that $r \neq 0$. Note the convergence
\begin{align*}
\Big| M_\beta & f_j (x) -  r^\beta \intav_{B(z, r)} |f(y)| \dy  \Big| \\
& \lesssim \frac{r_j^{\beta }}{r_j^d} \int_{\R^d} |f_j - f| +  \Big|\int_{\R^d} |f(y)| \big(\frac{r_j^\beta}{r_j^{d}} \chi_{B(z_j,r_j)} (y) - \frac{r^\beta}{r^{d}} \chi_{B(z,r)} (y) \big) \dy \Big| \to 0
\end{align*}
as $j \to \infty$. The first term goes to $0$ as $r_j \to r >0$ and $\| f_j - f \|_{L^1(\R^d)} \to 0$ as $j \to \infty$. The convergence of the second term may be seen by the dominated convergence theorem, as $f \in L^1$, $(z_j, r_j) \to (z,r)$ as $j \to \infty$ and $r_j, r >C$ for some constant $C$ and $j$ large enough. As $\| f_j - f \|_{L^p(\R^d)} \to 0$ for $1 \leq p \leq \frac{d}{d-1}$, then $\| M_\beta f_j - M_\beta f \|_{L^{r}(\R^d)} \to 0$ as $j \to \infty$ for some $r>\frac{d}{d-\beta}$ and therefore there is a subsequence $\{M_{\beta} f_{j_k}\}_{k \in \N}$ converging to $M_\beta f$ almost everywhere as $k \to \infty$, so $B(z,r) \in \CB_x^\beta (f)$.

We conclude the proof observing that, by contradiction, the case $r=0$ does not happen for $x \in E$. To see this, define the set $A_{j}=\{y\in E: M(f_j-f)(y)>1\}$. If $|\{j \in \N :  x\notin A_j\}|=\infty$ then going through a subsequence, if necessary,
$$
M_{\beta}f_{j_k}(x)\leq r^{\beta}_{j_k}M(f_{j_k}-f)+\frac{r^{\beta}_{j_k}}{r_{j_k}^{d}}\int_{B_{z,r_{j_k}}}|f| \to 0+0\times |f(x)|=0
$$
by the Lebesgue differentiation theorem, which is a contradiction. Otherwise, if $|\{j \in \N : x\notin A_j\}| < \infty$ then 
$$
x\in A:=\bigcup_{j_0\geq1}\bigcap_{j\geq j_0}A_{j},
$$
which is a measure zero set as a consequence of the weak (1,1) inequality for the maximal operator $M$ and the hypothesis $\| f_j - f \|_{L^1(\R^d)} \to 0$.
\end{proof}

In the case of $M_\beta^c$, the family of good balls $\CB_x^\beta$ is just determined by the family of good radii $\CR_x^\beta$. Of course, Lemma \ref{lemma:limit of good balls} continues to hold in this case, where $z=x$ and $z_j=x$ for all $j \in \N$.




\subsection{The derivative of $M_\beta$}

In order to understand the weak derivative $\nabla M_\beta f$, it is useful to recall the concept of approximate derivative. A function $f: \mathbb{R} \to \mathbb{R}$ is said to be {\it approximately differentiable} at a point $x_0\in \mathbb{R}$ if there exists a real number $\alpha$ such that, for any $\varepsilon >0$, the set
$$A_{\varepsilon} = \left\{ x \in \mathbb{R} : \ \frac{|f(x) - f(x_0) - \alpha(x- x_0)|}{|x-x_0|} < \varepsilon \right\}$$
has $x_0$ as a density point. In this case, the number $\alpha$ is called the \textit{approximate derivative} of $f$ at $x_0$ and it is uniquely determined. It follows directly from the definition that if $f$ is differentiable at $x_0$ then it is approximately differentiable at $x_0$, and the classical and approximate derivatives coincide. In the absence of differentiability, if the weak derivative of $f$ exists it also coincides with the approximate derivative \cite[Theorem 6.4]{EvansGariepy}.

Haj\l asz and Maly \cite{HM2010} showed that $M^c_0 f$ is approximate differentiable, and their arguments easily adapt to the non-centered maximal operator and to the fractional  setting. Moreover, the boundedness
$$
\| \nabla M_\beta f \|_q \leq C \| \nabla f \|_1
$$
for $1 \leq \beta < d$ \cite{KS2003} and $\beta \in (0,1)$ if $f$ is radial \cite{LM2017} implies that $M_\beta f$ is weakly differentiable in those cases and therefore its weak derivative equals to its approximate derivative, leading to the following lemma.
\begin{lemma}[Derivative of the maximal function \cite{LM2017}]\label{lemma:derivative Mbeta}
Let $f \in W^{1,1}(\R^d)$ and  $x\in\R^{d}$. Then, for all $B=B(z,r) \in \CB_x^\beta$, we have that
\begin{itemize}
\item[(i)] If $1 \leq \beta < d$, then $ M_{\beta}f$ is weakly differentiable and for almost every $x \in \R^d$ its weak derivative $\nabla  M_{\beta}f$ satisfies 
$$
\nabla  M_\beta f(x) = r^\beta \intav_{B} \nabla |f|(y) \dy
$$
and the same holds for $M_\beta^c f$.
\item[(ii)] If $\beta\in(0,1)$ and $f$ is a radial function, then  $ M_{\beta}f$ is differentiable a.e., and for almost every $x \in \R^d$ its derivative $\nabla  M_{\beta}f$ satisfies 
$$
\nabla  M_\beta f(x) = r^\beta \intav_{B} \nabla |f|(y) \dy.
$$
\end{itemize}
We call this identity Luiro's formula.
\end{lemma}

The value of the approximate derivative of $M_\beta f$ is a simple computation which can be obtained arguing as in \cite{HM2010} or \cite{LM2017}, and has its roots in the work of Luiro \cite{Luiro2007}. The stronger statement in (ii) regarding the a.e. differentiability of $M_\beta f$ in the radial case is a consequence of the one-dimensional result of Carneiro and the second author \cite{CM2015}, who showed that for $d=1$, the maximal function $M_\beta f$ is absolutely continuous and therefore differentiable almost everywhere in the classical sense; this extends to higher dimensions when acting on radial functions.




\subsection{A Brézis--Lieb type reduction}\label{ae+BrezisLieb}

In order to prove both Theorem \ref{thm:beta bigger than 1} and \ref{thm:beta smaller than 1}, we will show that for any $f \in W^{1,1}(\R^d)$ and $\{f_j\}_{j \in \N}$ sequence of functions in $W^{1,1}(\R^d)$ such that $\| f_j - f\|_{W^{1,1}(\R^d)} \to 0$ as $j \to \infty$, then
\begin{equation}\label{eq:goal}
\| \nabla M_\beta f_j - \nabla M_\beta f \|_{L^{d/(d-\beta)}(\R^d)} \to 0 \quad \textrm{as $j \to \infty$}.
\end{equation}
The classical Brézis--Lieb lemma \cite{BL1976} reduces the proof of \eqref{eq:goal} to showing that
$$
\int_{\R^d} |\nabla M_\beta f_j|^{\frac{d}{d-\beta}} \to \int_{\R^d} |\nabla M_\beta f|^{\frac{d}{d-\beta}} \quad \textrm{as $j \to \infty$}
$$
provided the almost everywhere convergence
\begin{equation}\label{eq:a.e. derivatives}
\nabla M_\beta f_j (x) \to \nabla M_\beta f(x) \quad \textrm{a.e. \:\: as $j \to \infty$}
\end{equation}
holds.

The rest of this section is devoted to show \eqref{eq:a.e. derivatives}, which is the content of the forthcoming Lemma \ref{lemma:ae convergence derivatives}.




\subsection{Almost everywhere convergence of the derivatives}

In order to show \eqref{eq:a.e. derivatives} we extend to higher dimensions and to the fractional case the strategy of Carneiro, Pierce and the second author \cite{CMP2017}. Their arguments do not straightforward generalise to higher dimensions due to the lack of uniform convergence of $M_\beta f_j$ to $M_\beta f$ (which holds for $d=1$ and $W^{1,1}(\R)$-functions).

In view of the representation of the derivative of $M_\beta$ in Lemma \ref{lemma:derivative Mbeta}, it is useful to note that convergence of $f_j$ to $f$ in $W^{1,1}$ implies convergence of their modulus. A proof of this functional analytic result is provided below for completeness as we could not find it in the literature. This fact was implicitly used in the work of Luiro \cite{Luiro2007}, to whom we are grateful for a helpful conversation regarding a step in the proof. It is noted that the one-dimensional version of this result has a slightly simpler proof based on the fundamental theorem of calculus; see \cite[Lemma 14]{CMP2017}.

\begin{lemma}\label{lemma:convergence modulus in W11}
Let $f \in W^{1,1}(\R^d)$ and $\{f_j\}_{j \in \N} \subset W^{1,1}(\R^d)$ be such that $\| f_j - f \|_{W^{1,1}(\R^d)} \to 0$ as $j \to \infty$. Then $\| |f_j| - |f| \|_{W^{1,1}(\R^d)} \to 0$ as $j \to \infty$.
\end{lemma}

\begin{proof}
Of course  $\| |f_j|- |f|\|_{L^1(\R^d)}  \leq \|f_j-f\|_{L^1(\R^d)} \to 0$ follows from the triangle inequality. 
To see that $\| \nabla |f_j| - \nabla |f| \|_{L^1(\R^d)} \to 0$, define the sets $X_j:= \{ x \in \R^d : f_j(x) >0 \}$, $Y_j:= \{ x \in \R^d : f_j(x) < 0 \}$ and $Z_j:= \{ x \in \R^d : f_j(x) = 0 \}$ for all $j \in \N$, and let $X, Y$ and $Z$ be defined similarly with respect to $f$. It then suffices to show the convergence on each of the nine subsets obtained by intersecting $X_j, Y_j, Z_j$ with $X,Y,Z$. Note that on $X_j \cap X$, $Y_j \cap Y$ and $Z_j \cap Z$, one has $|\nabla |f_j| - \nabla |f|| =|\nabla f_j - \nabla f|$ and therefore the convergence on those sets follows from the hypothesis $\| \nabla f_j - \nabla f \|_{L^1(\R^d)} \to 0$.

On $X_j \cap Z$ and $Y_j \cap Z$, one should note that $\nabla f = \nabla |f| = 0$ except for a set of measure zero. Indeed, if $I \subset Z$ has positive measure, one has $f(x)=|f(x)|=0$ on $I$ and therefore $\nabla f = \nabla |f| = 0$. Then $|\nabla |f_j| - \nabla |f|| =|\nabla f_j - \nabla f|$ a.e. on $X_j \cap Z$ and $Y_j \cap Z$ and the convergence on such sets follows again simply by the hypothesis $\| \nabla f_j - \nabla f \|_{L^1(\R^d)} \to 0$. The terms corresponding to $ Z_j \cap X$ and $Z_j \cap Y$ follow in a similar manner.

On $X_j \cap Y$,
$$
\int_{X_j \cap Y} |\nabla |f_j| - \nabla |f| | = \int_{X_j \cap Y} |\nabla f_j + \nabla f | \leq  \int_{\R^d} |\nabla f_j - \nabla f |  + \int_{X_j \cap Y} 2 | \nabla f |.
$$
The first term goes to $0$ as $j \to \infty$, as by hypothesis $\| \nabla f_j - \nabla f \|_{L^1(\R^d)} \to 0$. 

To show that second term goes to $0$, it suffices to see that $|X_j \cap Y| \to 0$ as $j \to \infty$. Indeed, assume that this assumption holds and, for a contradiction, that there exists a subsequence $j_{k}$ and $c>0$ such that
$$
\lim_{k \to \infty} \int_{X_{j_k} \cap Y} 2 |\nabla f| \geq c.
$$
As it is assumed that $|X_{j_k} \cap Y| \to 0$, there exists a further subsequence $j_{k_{\ell}}$ for which $\chi_{X_{j_{k_\ell}} \cap Y} \to 0$ a.e., and thus the dominated convergence theorem yields
$$
\lim_{\ell \to \infty} \int_{X_{j_{k_\ell}}} 2 |\nabla f| = 0,
$$
a contradiction. Finally, to show that $|X_j \cap Y| \to 0$, for any given $\varepsilon>0$, let $\delta>0$ be such that
$$
|A_\delta|:=|\{ x \in \R^d : 0 < f(x) \leq \delta\}| \leq \varepsilon/2.
$$
The set $\{x \in \R^d: f_j(x) < 0 \:\: \text{and} \:\: f(x) < \delta \}$ is contained in $\{x \in \R^d: |f(x)-f_j(x)| > \delta\}$, and the measure of the latter converges to 0 as $j \to \infty$ by hypothesis (convergence in $L^1$ implies convergence in measure). Thus, there exists $j_0 \in \N$ large enough so that
$$
|B_\delta|:=|\{ x \in \R^d: f_j(x) < 0 \:\: \text{and} \:\: f(x) < \delta  \}| \leq \varepsilon/2
$$
for all $j \geq j_0$. As $X_j \cap Y := A_\delta \cup B_\delta$, the result follows from combining the two previous displays. The term corresponding to $Y_j \cap X$ follows analogously, and the proof is then concluded.
\end{proof}

We now have all the necessary ingredients to prove \eqref{eq:a.e. derivatives}. The proof is a minor variant of its one-dimensional counterpart in \cite[Lemma 15]{CMP2017}; full details are given below for completeness.

\begin{lemma}\label{lemma:ae convergence derivatives}
Let $f \in W^{1,1}(\R^d)$ and $\{f_j\}_{j \in \N} \subset W^{1,1}(\R^d)$ be such that $\| f_j - f \|_{W^{1,1}(\R^d)} \to 0$ as $j \to \infty$. Then
\begin{equation}\label{eq:a.e. conv derivatives}
\nabla M_\beta f_j (x) \to \nabla M_\beta f(x) \quad \text{a.e.} \quad as \:\:j \to \infty
\end{equation}
if Luiro's formula holds for $M_\beta$, and the same holds for $M_\beta^c$.
\end{lemma}

\begin{proof}
Set $f_0=f$, and for every $j \geq 0$ let $E_j$ be the set of measure zero for which Lemma \ref{lemma:derivative Mbeta} fails for $f_j$. The set $E:=\cup_{j\geq 0} E_j$ continues to have measure zero. Let $F$ be the sets of measure zero for which Lemma \ref{lemma:limit of good balls} fails. That is, if $x\in F$ and $\{(z_j, r_j)\}_{j \in \N}$ is a sequence where $B(z_j, r_j) \in \mathcal{B}_x^\beta$, an accumulation point $(z,r)$ of $\{(z_j, r_j)\}_{j \in \N}$ does not necessarily satisfy $B(z,r) \in \mathcal{B}_x^\beta$. It then suffices to prove the desired result for $x \in D:=\R^d \backslash (E \cup F)$.

Given $x \in D$, there exist $\delta=\delta(x)>0$ and $N=N(x)<\infty$ such that $\CR_{x}^\beta \subset [\delta, N]$. We claim that there exists $j_0=j_0(x)$ such that $\CR_{x,j}^\beta \subset (\delta/2, 2N)$ for $j \geq j_0$. Otherwise, we may find a sequence $\{r_{j_k}\}_{k \geq 1} \subset [0, \delta/2] \cup [2N,\infty)$. If there exists a constant $C<\infty$ such that $\{r_{j_k}\}_{k \in \N} \subset [0, \delta/2] \cup [2N,C]$, the sequence $\{r_{j_k}\}_{k \in \N}$ admits a convergent subsequence $\{r_{j_{k_\ell}}\}_{\ell \in \N}$. By Lemma \ref{lemma:limit of good balls}, $\lim_{\ell \to \infty} r_{j_{k_\ell}} \in \CR_{x}^\beta$ but by construction this limit lies in $[0, \delta/2] \cup [2N,C]$, which is a contradiction. If one cannot find such a $C< \infty$,  there exists a subsequence $\{r_{j_{k_\ell}}\}_{\ell \in \N}$ such that $\lim_{\ell \to \infty} r_{j_{k_\ell}} = \infty$, which is again a contradiction by Lemma \ref{lemma:limit of good balls}.

Let $r_j \in \CR_{x,j}^\beta$ for $j \geq j_0$ and $z_j$ such that $B_j=B(z_j,r_j) \in \CB_{x,j}^\beta$. Using the above lower bound on $r_j$ and Lemma \ref{lemma:derivative Mbeta} one has
$$
|\nabla M_\beta f_j (x)| \lesssim  r_j^{\beta-d}  \int_{B_j} |\nabla| f_j|| \leq \delta^{\beta-d} ( \| \nabla |f_j| - \nabla |f| \|_{L^1(\R^d)} + \| \nabla |f| \|_{L^1(\R^d)}) \leq C
$$
for $j \geq \max\{j_0, j_1\}$, where $j_1$ is such that $\| \nabla |f_j| - \nabla |f| \|_{L^1(\R^d)} < \varepsilon$ for some $\varepsilon$. Then $\{\nabla M_\beta f_j (x)\}_{j \in \N}$ is a bounded sequence. Consider any convergent subsequence $\{\nabla M_{\beta} f_{j_k}(x)\}_{k \in \N}$. As the sequence $\{r_{j_k}\}_{k \in \N}$ is bounded, passing to a further subsequence one may assume that $(z_{j_{k_\ell}}, r_{j_{k_\ell}}) \to (z,r)$ as $\ell \to \infty$, where $B(z,r) \in \CB_{x}^\beta$ by Lemma \ref{lemma:limit of good balls}. By Lemma \ref{lemma:derivative Mbeta}
$$
\nabla M_\beta f_{j_{k_\ell}} (x) = r_{j_{k_\ell}}^\beta \intav_{B_{j_{k_\ell}}}  \nabla |f_{j_{k_\ell}}| \qquad \textrm{and} \qquad \nabla M_\beta f (x) = r^\beta \intav_{B(z,r)}  \nabla |f|.
$$
Then $\nabla M_\beta f_{j_{k_\ell}} (x) \to \nabla M_\beta f (x)$ as $\ell \to \infty$, as
\begin{align*}
\Big| & r_{j_{k_\ell}}^\beta  \intav_{B_{j_{k_\ell}}}  \nabla |f_{j_{k_\ell}}| -  r^\beta \intav_{B(z,r)}  \nabla |f| \Big| \\ & \:\: \lesssim  \frac{r_{j_{k_\ell}}^\beta}{r_{j_{k_\ell}}^d} \int_{\R^d} \big| \nabla | f_{j_{k_\ell}}| -\nabla  | f| \big|  
  + \int_{\R^d} |\nabla | f| | \Big( \frac{r_{j_{k_\ell}}^\beta}{r_{j_{k_\ell}}^d} \chi_{B_{j_{k_\ell}}}(y) -  \frac{r^\beta}{r^d} \chi_{B(z,r)}(y) \Big) \dy   \to 0
\end{align*}
as $\ell \to \infty$; the first term goes to $0$ by Lemma \ref{lemma:convergence modulus in W11} whilst the second term can be seen to go to $0$ by the dominated convergence theorem, as $f \in W^{1,1}$ and the radii $r_{j_{k_{\ell}}}$ are bounded below. Then, the original convergent subsequence $\{\nabla M_{\beta} f_{j_k}(x)\}_{k \in \N}$ converges to $\nabla M_\beta f(x)$ as $k \to \infty$. As this holds for any convergent subsequence $\{ \nabla M_\beta f_{j_k}(x)\}_{k \in \N}$ of $\{ \nabla M_\beta f_{j}(x)\}_{j \in \N}$, one has that $\nabla M_\beta f(x)$ is the unique accumulation point of $\{\nabla M_\beta f_j(x) \}_{j \in \N}$, and thus the result follows because such a sequence is bounded.
\end{proof}

\begin{remark}\label{remark:ae convergence Mbeta}
Note that the above proof also shows that, in particular, for any  $0<\beta < d$,
\begin{equation}\label{eq:a.e. convergence Mbeta}
M_\beta f_j(x) \to M_\beta f(x)
\end{equation}
a.e. on $\R^d$ as $j \to \infty$, provided $\| f_j - f \|_{W^{1,1}} \to 0$. Note that for $d=1$, or $d>1$ and $\beta \in (d-1,d)$ this is slightly easier due to the $L^\infty$ boundedness of $M_\beta$ for $f \in W^{1,1}(\R^d)$. The same holds for $M_\beta^c$.
\end{remark}




\subsection{A classical convergence result}

Finally, the following classical variant of the dominated convergence theorem will be used several times throughout the paper.

\begin{theorem}[Generalised Dominated Convergence Theorem]\label{thm:gdct}
Let $1 \leq p < \infty$ $f, g \in L^p(\R^d)$  and $\{f_j\}_{j \in \N}$ and $\{g_j\}_{j \in \N}$ be sequences of functions on $L^p(\R^d)$ such that
\begin{enumerate}
    \item[\textit{(i)}] $|f_j(x)| \leq |g_j(x)|$  a.e.,
    \item[(ii)] $f_j(x) \to f(x)$ and $g_j(x) \to g(x)$ a.e. as $j \to \infty$,
    \item[(iii)] $\| g_j - g \|_{L^p(\R^d)} \to 0$.
\end{enumerate}
Then $\| f_j - f \|_{L^p(\R^d)} \to 0$.
\end{theorem}

The proof of this theorem is standard and consists in two applications of Fatou's lemma; see for instance \cite[Chapter 4, Theorem 19]{royden2010real}.






\section{The case $1 \leq \beta < d$: Proof of Theorem \ref{thm:beta bigger than 1}}\label{sec:beta big}

This follows from a simple application of the Generalised Dominated Convergence Theorem together with the inequality \eqref{KS} and the a.e. convergences \eqref{eq:a.e. conv derivatives} and \eqref{eq:a.e. convergence Mbeta}.

Indeed, let  $f \in W^{1,1}(\R^d)$ and $\{f_j\}_{j \in \N} \subset W^{1,1}(\R^d)$ such that $\| f_j - f \|_{W^{1,1}(\R^d)} \to 0$ as $j \to \infty$.
Recall the inequality \eqref{KS} of Kinnunen and Saksman \cite{KS2003},
$$
|\nabla M_\beta f_j(x)| \leq M_{\beta-1}  f_j(x) \quad \textrm{for all $j > 0$},
$$
which holds for all $1 \leq \beta < d$ as $f_j \in L^r$ for $1 \leq r \leq \frac{d}{d-1}$. By Lemma \ref{lemma:ae convergence derivatives}, one has
\begin{equation*}
    \nabla M_{\beta} f_j \to \nabla M_{\beta} f \quad \textrm{a.e. \:\: as $j \to \infty$}.
\end{equation*}
By Remark \ref{remark:ae convergence Mbeta}
$$M_{\beta - 1} f_j \to M_{\beta -1}f \quad \textrm{a.e. \:\: as $j \to \infty$}$$
and, moreover, the sublinearity and boundedness of $M_{\beta-1}$ implies
$$
\|M_{\beta-1} f_j - M_{\beta-1} f \|_{L^{\frac{d}{d-\beta}}(\R^d)} \lesssim \| f_j - f \|_{L^{\frac{d}{d-1}}(\R^d)} \lesssim \| \nabla f_j - \nabla f \|_{L^1(\R^d)} \to 0 $$
as $j \to \infty$.

The hypothesis of Theorem \ref{thm:gdct} are then satisfied, yielding
$$
\| \nabla M_\beta f_j - \nabla M_\beta f \|_{L^{\frac{d}{d-\beta}}(\R^d)} \to 0 \quad \textrm{as $j \to \infty$},
$$
as desired.




\section{The case $0 < \beta < 1$ for radial functions: Proof of Theorem \ref{thm:beta smaller than 1}}\label{sec:beta small}

The proof strategy for Theorem \ref{thm:beta smaller than 1} consists in studying separately what happens inside and outside a large compact set $K$. The main difficulty relies in establishing convergence in $K$; the term corresponding to $K^c$  may be seen as an error term. This strategy was already used by the second author in the one dimensional case \cite{Madrid2017}. However, the techniques used therein to analyse $K$ and $K^c$ only continue to work in very specific situations, and we need to develop a new approach in higher dimensions to deal with the general situation.\footnote{As mentioned in the Introduction, the analysis on $K$ for $d=1$ in \cite{Madrid2017} only extends in a natural way to higher dimensions if $d-1 < \beta < d$; further details of this will be provided in the Appendix \ref{app: (d-1,d)}}

In order to overcome the higher dimensional obstacles, we make use of some fundamental observations that proved to be useful in establishing the bound 
\begin{equation}\label{eq:bound again}
\| \nabla M_\beta f \|_q \leq C(d,\beta) \| \nabla f\|_1
\end{equation}
for radial $f$ in \cite{LM2017}. We remark that in contrast to \cite{Madrid2017}, our analysis outside the compact set is rather general and continues to hold for general function, any dimension, the centered case and any $0 \leq \beta < d$ (including the classical Hardy--Littlewood maximal operator) provided the bound \eqref{eq:bound again} holds in each corresponding case. This will be appropriately discussed in Section \ref{subsec:outside compact}.




\subsection{Preliminaries}\label{subsec:preliminaries}


A trivial but important observation for the non-centered maximal function is that if $|\nabla M_\beta f(x) |\neq 0$ and $B \in \CB_x^\beta$, then $x \in \partial B_x$: as $B_x$ is an admissible ball for all $y \in B_x$, one would have $M_\beta f (x) \leq  M_\beta f(y)$ for all $y \in B_x$, so if $x$ lied in the interior of the ball, it would be a local minimum for $M_\beta f$ and therefore $\nabla M_\beta f(x)=0$. 

Arguing in a similar manner, if $f$ is a radial function, $|\nabla M_\beta f (x)| \neq 0$ and $B_x \in \CB_x^\beta$, the center of the ball $B$ must lie in the direction joining $x$ and the origin: otherwise, there is a point $y$ lying in the interior of $B_x$ with $|y|=|x|$ which by radiality satisfies $|\nabla M_\beta f (x)|=|\nabla M_\beta f(y)|$, and the previous argument would imply $|\nabla M_\beta f (y)| = 0$. Thus, if $B_x=B(z_x,r_x) \in B_x^\beta$, one has $z_x=c_x x$ for some constant $c_x$. However, by radiality and the argument just described, one must have $c_x \geq 0$, as otherwise $-x$ lies in the interior of $B_{x}$. Then we are left with two cases: either
\begin{equation}\label{type of balls}
B_{x} \subseteq B(0,|x|) \qquad \text{or} \qquad B_{x} \subset B(0,|x|)^c.
\end{equation}
The first case corresponds to $0 \leq c_x \leq 1$ and the second one to $c_x > 1$.

Next we shall recall two preliminary lemmas observed in \cite{LM2017} that will be useful to the proof of Theorem \ref{thm:beta smaller than 1}. The first one corresponds to a refinement of Kinnunen's pointwise estimate \eqref{eq:Kinnunen}. 
\begin{lemma}[Lemma 2.9 \cite{LM2017}]\label{lemma:ye}
Suppose that  $f\in W^{1,1}_{loc}(\R^d)$, $0 < \beta < d$ and $B_x \in \mathcal{B}^{\beta}_x$ for some $x\in\R^d \setminus\{0\}$ such that $B_x\subset B(0,|x|)$.  Then
\begin{equation*}
\bigg|\intav_{B_x}\nabla |f|(y)\,\dy\,\bigg|\,\leq\,\intav_{B_x}|\nabla f(y)|\frac{|y|}{|x|}\dy\,.
\end{equation*}
\end{lemma}

The second one is a refinement of the Kinnunen--Saksman inequality \eqref{KS}, which in fact is an implicit consequence of their proof. It is noted that this refinement also works for the centered maximal function - this will be used in Section \ref{subsec:outside compact}

\begin{lemma}[\cite{KS2003}]\label{lemma:refinement KS}
Suppose that $f \in W^{1,1}_{loc}(\R^d)$, $0 < \beta < d$ and $B_x \in \mathcal{B}_x^\beta$ for some $x \in \R^d$, and let $r_x$ denote the radius of $B_x$. Then
$$
\Big| r_x^\beta \intav_{B_x} \nabla |f| (y) \dy \Big| \leq C(d,\beta) r_x^{\beta-1} \intav_{B_x} |f(y)| \dy.
$$
\end{lemma}

In fact, Luiro and Madrid \cite[Lemma 2.7]{LM2017} obtained a further refinement which consists on an equality with a boundary term arising from integration-by-parts, although such a stronger statement will not be needed for the purposes of this paper.

\begin{remark}
The above lemmas continue to hold for $\beta=0$ if $x$ is such that $Mf(x)>f(x)$, which ensures $0 \notin \CR^0_x$.
\end{remark}





\subsection{Inside a compact set $K\subset \R^{d}, d>1$}

We first prove convergence inside a compact set $K$.

\begin{proposition}\label{prop:convergence compact}
Let $0<\beta < 1$, $f \in W^{1,1}(\R^d)$ and $\{f_j\}_{j \in \N} \subset W^{1,1}(\R^d)$ radial functions such that $\| f_j - f \|_{W^{1,1}(\R^d)} \to 0$. Then, for any compact set $K=\bar B(0,b)$,
\begin{equation}\label{goal in compact}
\| \nabla M_\beta f_j - \nabla M_\beta f \|_{L^q(K)} \to 0 \qquad \textrm{as \: $j \to \infty$,}
\end{equation}
where $q=d/(d-\beta)$.
\end{proposition}


To this end, we start establishing the desired result for the auxiliary operator
\begin{equation*}\label{eq:MI}
    M^{I}_\beta f_j(x) :=  \sup_{\bar B(z,r)\ni x, r\leq |x|/4}
\frac{r^\beta}{|B(z,r)|}\int_{B(z,r)}|f(y)|\dy.
\end{equation*}
In the case $\beta=0$ this operator was introduced by Luiro \cite{Luiro2017}. Its fractional counterpart was implicitly studied in \cite{LM2017}, and in particular
\begin{equation}\label{bound MI}
\| \nabla M_\beta^I f \|_{L^q(\R^d)} \lesssim \| \nabla f \|_{L^1(\R^d)}
\end{equation}
holds for radial $f$ from the analysis on the set $E_3$ in \cite{LM2017}.

The following lemmas will be crucial to analyse the convergence of $M_\beta^I$ at the derivative level.
\begin{lemma}\label{MI smallnes around 0}
Let $0<\beta < 1$, $f \in W^{1,1}(\R^d)$ and $\{f_j\}_{j \in \N} \subset W^{1,1}(\R^d)$ radial functions such that $\| f_j - f \|_{W^{1,1}(\R^d)} \to 0$. Then, for every $\epsilon>0$ there is a ball $B_1=B(0,a)$ such that 
$$
\|\nabla M^{I}_{\beta}f\|_{L^{q}(B_1)}<\epsilon \quad \text{and}\ \quad  \|\nabla f\|_{L^{1}(B_1)}<\epsilon,
$$
and, moreover,
$$
\|\nabla M^{I}_{\beta}f_j\|_{L^{q}(B_1)}<\epsilon \quad \text{and}\ \quad  \|\nabla f_j \|_{L^{1}(B_1)}<\epsilon 
$$
for all $j\geq j(\epsilon)$.
\end{lemma}

\begin{proof}[Proof of Lemma \ref{MI smallnes around 0}]
For every $\epsilon>0$ there exists a ball $B=B(0,\delta)$ such that 
$\|\nabla f\|_{L^{1}(B)}<\epsilon$ and $\|\nabla M^I_{\beta}f\|_{L^{q}(B)}<\epsilon$, also there exists $j(\epsilon)$ such that $\|\nabla f_j-\nabla f\|_{L^1{(\R^d)}}<\epsilon$ for all $j\geq j(\epsilon)$.  Then
$\|\nabla f_{j}\|_{L^{1}(B)}<2\epsilon $ for all $j\geq j(\epsilon)$. Moreover 
$$
\|\nabla M^I_{\beta}f_j\|_{L^{q}(\frac{2}{3}B)}\leq \|\nabla M^I_{\beta}(f_j\chi_{B})\|_{L^{q}(B)}\lesssim \|\nabla f_j\|_{L^{1}(B)}<2\epsilon
$$
for all $j\geq j(\epsilon)$, where the first inequality follows by the definition of $M^I$ and the second one from the boundedness \eqref{bound MI}. The conclusion is obtained choosing $B_1=\frac{2}{3}B$.
\end{proof}


\begin{lemma}\label{MI uniformity on rings}
Let $0\leq\beta < 1$, $f \in W^{1,1}(\R^d)$ and $\{f_j\}_{j \in \N} \subset W^{1,1}(\R^d)$ radial functions such that $\| f_j - f \|_{W^{1,1}(\R^d)} \to 0$. Then, for any annulus $A(a,b):=\{ x \in \R^d:  a \leq |x| \leq b\}$ with $0<a<b/3<\infty$ we have that $\|f_j-f\|_{L^{\infty}(A(a,b))}\to 0$ as $j\to\infty$ then $\|M^{I}_{\beta}f_j-M^{I}_{\beta}f\|_{L^{\infty}(A(2a,2b/3))}\to0$. 
\end{lemma}

\begin{proof}
Let $\tilde{f}, \tilde{f}_j: (0,\infty) \to \R$ be such that $f(x)=\tilde{f}(|x|)$ and $f_j(x)=\tilde{f}_j(|x|)$. Note that $\nabla f (x) = \tilde{f}'(|x|) \frac{x}{|x|}$ and $\nabla f_j (x) = \tilde{f}_j'(|x|) \frac{x}{|x|}$. By hypothesis one has that
\begin{equation}\label{eq:hyp radial convergence}
\int_{0}^\infty |\tilde{f}_j(t)-\tilde{f}(t)| t^{d-1}  \dt  \to 0 \quad \text{and} \quad \int_{0}^\infty |\tilde{f}'_j(t)-\tilde{f}'(t)| t^{d-1}  \dt \to 0
\end{equation}
as $j \to \infty$. Note that given $g \in W^{1,1}((0,\infty))$, by the Fundamental Theorem of Calculus,
$$
|g(x)| \leq |g(y)| + \int_{a}^b |g'(t)| \dt
$$
for any $x, y \in [a,b]$. Averaging over $y \in [a,b]$ one has
$$
|g(x)| \leq \frac{1}{b-a} \int_{a}^b |g(t)| \dt + \int_{a}^b |g'(t)| \dt.
$$
Applying this for $g=\tilde{f}_j- \tilde{f}$, it follows that for $a \leq |x| \leq b$,
$$
|f_j(x)-f(x)| \leq \frac{1}{(b-a) a^{d-1}} \int_a^b |\tilde{f}_j(t) - \tilde{f}(t)| t^{d-1}\dt + \frac{1}{a^{d-1}} \int_a^b |\tilde{f}_j ' (t) - \tilde{f}'(t)| t^{d-1} \dt
$$
and using \eqref{eq:hyp radial convergence} it follows that $\| f_j - f \|_{L^\infty(A(a,b))} \to 0$ as $j \to \infty$.

Finally, note that for $2a \leq |x| \leq 2b/3$, one has $M^I_\beta f = M^I_\beta (f \chi_{A(a,b)})$, as the admissible radii $r$ in the definition of $M^I_\beta$ satisfy $r \leq |x|/4$.  
\end{proof}


\noindent The next Lemma follows similarly to Lemma \ref{lemma:ae convergence derivatives}.
\begin{lemma}\label{MI point conver}
Let $0<\beta < 1$, $f \in W^{1,1}(\R^d)$ and $\{f_j\}_{j \in \N} \subset W^{1,1}(\R^d)$ radial functions such that $\| f_j - f \|_{W^{1,1}(\R^d)} \to 0$. Then
\begin{equation}\label{goal in compact}
 \nabla M^{I}_\beta f_j(x) \to \nabla M^{I}_\beta f(x) \quad  \ a.e\ \quad \textrm{as \: $j \to \infty$}.
\end{equation}
\end{lemma}


In view of the previous lemmas, the desired convergence result for $M^I_\beta$ on radial functions can be obtained using the one-dimensional arguments in \cite{Madrid2017}. Those arguments cannot be extended to $M_\beta$, as the full maximal operator lacks the uniform convergence obtained in Lemma \ref{MI uniformity on rings}. This is in contrast with $d=1$, where uniform convergence follows for $M_\beta$ for convergent sequences of functions in $W^{1,1}(\R)$. 

\begin{proposition}\label{prop:convergence compactMI}
Let $0<\beta < 1$, $f \in W^{1,1}(\R^d)$ and $\{f_j\}_{j \in \N} \subset W^{1,1}(\R^d)$ radial functions such that $\| f_j - f \|_{W^{1,1}(\R^d)} \to 0$. Then, for any compact set $K=\bar B(0,b)$,
\begin{equation}\label{goal in compact}
\| \nabla M^{I}_\beta f_j - \nabla M^{I}_\beta f \|_{L^q(K)} \to 0 \qquad \textrm{as \: $j \to \infty$,}
\end{equation}
where $q=d/(d-\beta)$.
\end{proposition}

\begin{proof}
By Lemma \ref{MI smallnes around 0} it suffices to show the convergence for any annulus $A(a,b)$ with $0<a<b/3<\infty$. By Lemma \ref{MI point conver} and the dominated convergence theorem, it suffices to show that there exist a constant $C>0$ and $j_0\in\N$ such that
$$
|\nabla M_{\beta}^If_{j}(x)|\leq C \qquad \text{for all $\, x\in A(a,b) \, $ and all $ \,j \geq j_0\,$,}
$$
as constants are integrable on bounded domains.

To this end, let  $ C_ A,C_{A,j}>0$ be such that
$$
\inf_{x \in A(a,b)} M_\beta^I f (x) = C_A \qquad \textrm{and} \qquad  \inf_{x \in A(a,b)} M_\beta^I f_j (x) = C_{A,j};
$$
note that these constants always exist provided $f$ is not identically 0. By Lemma \ref{MI uniformity on rings} one has $\| M_\beta^I f_j - M_\beta^I f \|_{L^\infty(A(a,b))} \to 0$ as $j \to \infty$, so there exists $j_1(A) \in \N$ such that $C_{A,j} > C_{A}/2$ for all $j>j_1(A)$. For each $x \in A(a,b)$, let $B_{x,j}:=B(z_{x,j}, r_{x,j}) \in \CB_{x,j}^\beta$. Then
\begin{equation}\label{bound xx}
C_{A}/2 \leq  r_{x,j}^\beta  \intav_{B_{x,j}} |f_j| \leq r_{x,j}^{\beta} \| f_j \|_{L^\infty(A(a,b))} \lesssim r_{x,j}^{\beta} \| f \|_{L^\infty(A(a,b))}
\end{equation}
for $j > j_0:= \max \{j_1(A), j_2(A)\}$ where $j_2(A) \in \N$ is large enough so that $\| f_j - f \|_{L^\infty(A(a,b))}\leq \| f \|_{L^\infty(A(a,b))}$, which holds by Lemma \ref{MI uniformity on rings}. As $\beta>0$, one has the uniform lower bound $r_{x,j} \gtrsim (C_A)^{1/\beta}=: \bar C_A>0$.

This uniform lower bound on the radius together with Lemmas \ref{lemma:derivative Mbeta} and \ref{lemma:convergence modulus in W11} yield the desired bound
\begin{align*}\label{uniform bound on K}
|\nabla M_\beta^I f_j (x)| & \leq \Big| r_{x,j}^\beta  \intav_{B_{x,j}} \nabla |f_j|(y) \dy \Big| \nonumber\\
 & \lesssim \frac{1}{(\bar C_{A} )^{d-\beta}} \big(  \| \nabla  |f_j| - \nabla |f| \|_{L^1(\R^d)} +  \| \nabla |f| \|_{L^1(\R^d)} \big) \nonumber\\
 & \lesssim 1
\end{align*}
for all $x \in A(a,b)$ and all $j> j_0$. 

\end{proof}

\begin{remark}
The arguments used for to prove Proposition \ref{prop:convergence compactMI} continue to work for the centered version of $M^I_\beta$.
\end{remark}


Now we are in position to obtain our desired convergence result for the full $M_{\beta}$ at the derivative level on compact sets.

\begin{proof}[Proof of Proposition \ref{prop:convergence compact}]
Set $f_0=f$, and let $E_j$ be the set of measure zero for which Lemma \ref{lemma:derivative Mbeta} fails for $f_j$. The set $E:=\cup_{j\geq 0} E_j$ continues to have measure zero. Let $F$,  $G$ and $H$ be the set of measure zero for which Lemmas \ref{lemma:limit of good balls}, \ref{lemma:ae convergence derivatives} and Remark \ref{remark:ae convergence Mbeta} fail respectively. It then suffices to show \eqref{goal in compact} for $K$ replaced by $\widetilde{K}:=K \backslash(E \cup F \cup G \cup H)$, which for ease of notation is relabelled as $K$.

For all $j>0$ we have $K=K^{0}_{j}\cup U^{}_{j}\cup V^{}_{j}\cup W^{}_{j}$, where 
$K_{j}^{0}=\{x\in K: \nabla M_{\beta}f_{j}(x)=0\}$ and
$$U^{}_{j}=\{x\in K\setminus{K_{j}^{0}}: \exists \ B_{x,j} \in \CB^\beta_{x,j} \ \ \text{with}\ \  r_{x,j}> |x|/4\ \ \text{and}\ \  B_{x,j}\subset B(0,|x|)^{c}\},$$
$$V^{}_{j}=\{x\in K\setminus{K_{j}^{0}} : \exists \ B_{x,j} \in \CB^\beta_{x,j} \ \ \text{with}\ \ r_{x,j}> |x|/4\  \ \text{and}\ \ B_{x,j}\subset B(0,|x|)\}$$
and
$$
W^{}_{j}=\{ x \in K \setminus K_{j}^{0} : 
M_{\beta}f(x)=M^{I}_{\beta}f(x)\}.$$
The definitions of $U_j$ and $V_j$ are motivated by the two types of balls that one needs to consider when $|\nabla M_\beta f_j (x)| \neq 0$: see the discussion at the beginning of Section \ref{subsec:preliminaries} and display \eqref{type of balls}. The additional constraint $r_{x,j} > |x|/4$ is included because the case of small radii has already been analysed via the operator $M_\beta^I$. Define the functions
$$
u_j(x):= \int_{\R^{d}}|\nabla|f_{j}|(y)|\frac{\chi_{B(0,|y|)}(x)}{|y|^{d}}\dy,
$$
and
$$
v_j(x):= \frac{1}{|x|^{d}}\int_{B(0,|x|)}|\nabla f_{j}(y)|\frac{|y|}{|x|}\dy.
$$
By Lemma \ref{lemma:derivative Mbeta},
\begin{equation}\label{eq:q-1 out}
    |\nabla M_\beta f_j(x)|^q \leq \frac{1}{(\omega_d)^{q-1}} \| \nabla |f_j| \|_{1}^{q-1} \Big|  \intav_{B_{x,j}} \nabla |f_j|(y) \dy \Big| \quad \textrm{for all \: $x \in K$}.
\end{equation}
Note that as $|\nabla M_\beta f_j(x)| \neq 0$ on $U_j \cup V_j$, the good balls $B_{x,j} \in \CB_{x,j}^\beta$ are of the type described in the previous subsection: $x \in \partial B_{x,j}$ and the center of $B_{x,j}$ belongs to the line joining $x$ and the origin; this features in the following bounds on $U_j$ and $V_j$.

For every $x\in U_j$, if $y\in B_{x,j}$ one has $r_{x,j}\geq|y|-|x|\geq |y|-4r_{x,j}$ and $|x|\leq |y|$. Then
$$
\Big| \intav_{B_{x,j}}\nabla|f_{j}|(y) \dy \Big|
 \leq \frac{5^d}{\omega_d}\int_{\R^{d}}|\nabla|f_{j}|(y)|\frac{\chi_{B(0,|y|)}(x)}{|y|^{d}}\dy = \frac{5^d}{\omega_d} u_j(x) \quad \text{on}\:\: U_j.
$$
For every $x \in V_j$, one has $|x|/4< r_{x,j} \leq  2|x|$ and Lemma \ref{lemma:ye} then yields
\begin{equation*}
\Big| \intav_{B_{x,j}}\nabla|f_{j}|(y) \dy \Big| \leq \frac{4^d}{\omega_d |x|^d} \int_{B(0,|x|)}|\nabla f_{j}(y)|\frac{|y|}{|x|}\dy = \frac{4^d}{\omega_d} v_j(x) \quad \text{on}\:\: V_j.
\end{equation*}
Using \eqref{eq:q-1 out} in $U_j \cup V_j$ and the previous estimates, for all $j >0$,
\begin{equation}\label{K conclusion}
|\nabla M_{\beta}f_{j}(x)|^{q}\lesssim \| \nabla |f_j| \|_1^{q-1} \Big( u_{j}(x)+v_{j}(x) \Big) +|\nabla M_\beta^I f_j (x)|^q  
\quad  \text{on}\:\:  K .
\end{equation}
The desired result will follow from an application of the generalised dominated convergence theorem (Theorem \ref{thm:gdct}) for functions on $L^1$. Indeed, a successful application of that theorem would yield
$$
\| |\nabla M_\beta f_j|^q - |\nabla M_\beta f|^q \|_{L^1(K)} \to 0 \quad \textrm{as $j \to \infty$},
$$
and consequently
$$
\int_{K} |\nabla M_\beta f_j|^q \to \int_{K} |\nabla M_{\beta} f|^q \quad \textrm{as $j \to \infty$}.
$$
Convergence on $L^q(K)$ would now follow from the Brézis--Lieb lemma (see Subsection \ref{ae+BrezisLieb}). Therefore, it suffices to verify the hypothesis of Theorem \ref{thm:gdct} with the sequences involved in \eqref{K conclusion}.

Concerning the left-hand-side, the estimate $\|\nabla M_\beta f \|_q \lesssim \| \nabla f \|_1$ in \cite{LM2017} implies that the sequence $\{ |\nabla M_\beta f_j(x)|^q \}_{j \in \N}$ is on $L^1(K)$. Moreover, Lemma \ref{lemma:ae convergence derivatives} ensures that $|\nabla M_\beta f_j|^q \to |\nabla M_\beta f|^q$ a.e. as $j \to \infty$, satisfying the desired hypothesis.

Concerning the right-hand-side, by Lemma \ref{MI point conver} and Proposition \ref{prop:convergence compactMI} one has that
\begin{align}\label{w_j conv}
   |\nabla M_\beta^I f_j(x)| \to |\nabla M_\beta^I f(x)| \quad & \textrm{and} \quad \| \nabla M_\beta^I f_j - \nabla M_\beta^I f \|_{L^q(K)} \to 0
   \intertext{as $j \to \infty$, so it suffices to show}
   \label{u_j conv}
    u_j(x) \to u(x) \quad  & \textrm{and} \quad \| u_j - u \|_1 \to 0 \quad \textrm{as $j \to \infty$}, \\
\label{v_j conv}
    v_j(x) \to v(x) \quad & \textrm{and} \quad \| v_j - v \|_1 \to 0 \quad \textrm{as $j \to \infty$}
\end{align}
where $u$ and $v$ are defined analogously to $u_j$ and $v_j$ respectively but with $f_j$ replaced by $f$.
Indeed, Lemma \ref{lemma:convergence modulus in W11} ensures that $\| \nabla |f_j| - \nabla |f| \|_1 \to 0$ as $j \to \infty$, so together with \eqref{w_j conv}, \eqref{u_j conv} and \eqref{v_j conv} this implies that the right-hand-side on \eqref{K conclusion} converges a.e. and on $L^1$, as desired for the application of Theorem \ref{thm:gdct}.

The rest of the proof is devoted to verify \eqref{u_j conv} and \eqref{v_j conv}.




\subsubsection{The case of $u_j$}
For any $x \neq 0$, one trivially has
\begin{align*}
|u_{j}(x)-u(x)| & \leq  \int_{\R^{d}}|\nabla|f_{j}|(y) - \nabla |f| (y)|\frac{\chi_{B(0,|y|)}(x)}{|y|^{d}}\dy \\
& \leq  \frac{1}{|x|^{d}}\| \nabla|f_{j}|-\nabla|f|\|_{1} \to 0 \quad \textrm{as $j \to \infty$}
\intertext{as $|y|\geq |x|$, so $u_j \to u $ a.e. as $j \to \infty$. Moreover, by Fubini's theorem}
\|u_{j}-u \|_{1} & \leq \int_{\R^d} |\nabla |f_j| (y) - \nabla |f|(y)| \int_{\R^d} \frac{\chi_{B(0,|y|)} (x) }{|y|^d} \dx \dy \\
& \lesssim
\|\nabla|f_{j}|-\nabla|f|\|_{1}\to0 \quad \textrm{as $j \to \infty$}.
\end{align*}




\subsubsection{The case of $v_j$}

Similarly, for any $x \neq 0$,
\begin{align*}
|v_{j}(x)-v(x)| & \leq   \frac{1}{|x|^{d}} \int_{B(0,|x|)} |\nabla f_j(y) - \nabla f(y)| \frac{|y|}{|x|} \dy \\
 & \leq  \frac{1}{|x|^{d}}\|\nabla f_{j}-\nabla f\|_{1} \to 0 \quad  \textrm{as $j\to\infty$}.
\intertext{Moroever, by Fubini's theorem and a change to polar coordinates one has}
\|v_{j}-v \|_{1}&\leq   \int_{\R^d} |\nabla f_{j}(y)-\nabla f(y)| |y| \int_{B(0, |y|)^c} |x|^{-d-1} \dx \dy  \\
&\lesssim \|\nabla f_{j} -\nabla f \|_{1} \to 0 \quad \textrm{as $j \to \infty$},
\end{align*}
as desired. This concludes the proof of Proposition \ref{prop:convergence compact}.






\end{proof}




\subsection{Smallness outside a compact set $3K$}\label{subsec:outside compact}

In order to conclude the proof of Theorem \ref{thm:beta smaller than 1}, it suffices to show smallness outside a compact set. Our argument relies on Lemma \ref{lemma:refinement KS}, and therefore continues to work for the case $\beta=0$, the centered maximal function $M_\beta^c$ and does not require any radial hypothesis on the functions.

\begin{proposition}\label{prop:smallness}
Let $0	\leq  \beta < d$, $f \in W^{1,1}(\R^d)$ and $\{f_j\}_{j \in \N} \subset W^{1,1}(\R^d)$ such that $\| f_j - f \|_{W^{1,1}(\R^d)} \to 0$. Assume that $\mathcal{M}_\beta \in \{M_\beta, M_\beta^c\}$ satisfies
\begin{equation}\label{sobolev boundedness hyp}
\| \nabla \mathcal{M}_\beta f \|_{q} \leq \| \nabla f \|_1,
\end{equation}
where $q=d/(d-\beta)$. Then, for any $\varepsilon>0$ there exists a compact set $K$ and $j_\varepsilon>0$ such that
$$
\| \nabla \CM_\beta f_j - \nabla \CM_\beta f \|_{L^q( ( 3K)^c )} < \varepsilon
$$
for all $j \geq j_\varepsilon$.
\end{proposition}

The above lemma may be applied in our case as the bound \eqref{sobolev boundedness hyp} is satisfied for the non-centered fractional maximal function $M_\beta$ acting on radial functions. As is mentioned above, it is remarked that it would also apply to the centered case, to general functions and to $\beta=0$ provided the hypothetical endpoint Sobolev bound \eqref{sobolev boundedness hyp} holds in such cases.

\begin{proof}
Let $1 < p <\frac{d}{d-1}$ and $r$ be such that $\frac{1}{r}=\frac{1}{p} - \frac{\beta}{d}$. As $f_j, f \in W^{1,1}$, one has $f_j, f \in L^p$, and by the boundedness of $\CM_\beta$  one has
\begin{equation}\label{eq:boundedness hyp}
\| \CM_\beta f \|_r \lesssim \| f \|_p.
\end{equation}
Given $\varepsilon>0$, let $K$ be a compact set satisfying
\begin{equation}\label{smallness conditions}
    \int_{K^c} |f| < \varepsilon, \:\: \int_{K^c} |\nabla f| < \varepsilon, \:\:
\int_{K^c} |\CM_\beta f|^r < \varepsilon^r \:\: \textrm{and} \:\: \int_{K^c} |\nabla \CM_\beta f|^q < (\varepsilon/2)^q
\end{equation}
for some $r>q$; note that the two last conditions follow from \eqref{eq:boundedness hyp} and the hypothesis \eqref{sobolev boundedness hyp}. Moreover, let $j_\varepsilon>0$ be such that
\begin{equation}\label{jepsilon}
\| f_j - f \|_{L^1(\R^d)} < \varepsilon \qquad \textrm{and} \qquad \| \nabla f_j - \nabla f \|_{L^1(\R^d)} < \varepsilon
\end{equation}
for all $j \geq  j_\varepsilon$.

For every $j \geq j_\varepsilon$ write $(3K)^{c}=Y^{j}_{1}\cup Y^{j}_{2},$ where $Y^{j}_{1}:=\{x \in (3K)^{c} : K\cap B_{x,j}=\emptyset\}$ and $Y^{j}_{2}=(3K)^{c}\setminus Y^{j}_{1}$. By the triangle inequality and the last condition in \eqref{smallness conditions} it suffices to show
$$
\int_{(3K)^c} |\nabla \CM_\beta f_j|^{q} < (\varepsilon/2)^q \qquad \textrm{for all $j \geq j_\varepsilon$.}
$$ 

On $Y_1^j$ one may replace $f_j$ by $f_j \chi_{\R^d \backslash K}$. Using \eqref{sobolev boundedness hyp}, \eqref{smallness conditions} and \eqref{jepsilon},
\begin{align*}
\int_{Y_1^j} |\nabla \CM_\beta f_j(y)|^q \dy & \leq \int_{\R^d} |\nabla \CM_\beta (f_j \chi_{\R^d \backslash K})(y)|^q \dy  \\
&\lesssim \| \nabla (f_j \chi_{\R^d \backslash K}) \|_{L^1(\R^d)}^q \\
& \lesssim \|(\nabla  f_j)  \chi_{\R^d \backslash K} \|_{L^1(\R^d)}^q \\
& \leq \|\nabla  f_j  - \nabla f \|_{L^1(\R^d)}^q +\|(\nabla  f)  \chi_{\R^d \backslash K} \|_{L^1(\R^d)}^q \\
& \lesssim 2 \varepsilon^q
\end{align*}
for all $j \geq j_\varepsilon$.

If $x \in Y_2^j$ one has $r_{x,j}>|x|/3$. This and Lemma \ref{lemma:refinement KS} imply\footnote{Note that for $\beta=0$, if $x \in Y_2^j$ then $0 \not \in \CR_x^0$, and Lemma \ref{lemma:refinement KS} can safely be applied in this case.}
$$
| \nabla \CM_{\beta}f_j(x)|  \leq \frac{C(d,\beta)}{r_{x,j}} \CM_{\beta}f_j(x) \leq  \frac{3C(d,\beta)}{|x|}\CM_{\beta}f_j(x).
$$
For $p$ and $r$ as above, note that $r>q$ and $\frac{qr}{r-q}>d$. Then, by Hölder's inequality, \eqref{eq:boundedness hyp}, \eqref{smallness conditions} and \eqref{jepsilon},
\begin{align}
\int_{Y_2^j}\left| \nabla \CM_{\beta}f_j(x)\right|^{q}\dx & \lesssim  \Big(\int_{Y_2^j}(\CM_{\beta}f_j)^{r}\Big)^{\frac{q}{r}}\Big(\int_{Y_2^j} |x|^{\frac{-rq}{(r-q)}} \dx \Big)^{\frac{r-q}{r}} \label{x integral} \\
&\lesssim   \left( \Big(\int_{Y_2^j}(\CM_{\beta}(f_j-f))^{r}\Big)^{\frac{1}{r}} + \Big(\int_{Y_2^j}(\mathcal{M}_{\beta}f)^{r}\Big)^{\frac{1}{r}}  \right)^{q}\notag \\
&\lesssim (\|f_j - f\|_{L^p(\R^d)} + \varepsilon)^{q} \notag \\
&\leq (2\epsilon)^q \notag
\end{align}
for all $j \geq j_{\varepsilon}$, as the values of $q$ and $r$ ensure that the second integral in \eqref{x integral} is uniformly finite provided $K$ contains the unit ball. Reverse engineering the choice of $\varepsilon$ in \eqref{smallness conditions} and \eqref{jepsilon} concludes the proof.
\end{proof}




\subsection{Concluding the argument: Proof of Theorem \ref{thm:beta smaller than 1}}

This is now a simple consequence of Propositions \ref{prop:convergence compact} and \ref{prop:smallness}. Given $\varepsilon>0$, by Proposition \ref{prop:smallness} there exist a compact set $K$ and $j_{\varepsilon,1}>0$ such that
$$
\| \nabla M_\beta f_j - \nabla M_\beta f \|_{L^q((3K)^c)} < \varepsilon/2
$$
for all $j \geq j_{\varepsilon,1}$. As $3K$ is itself a compact set, Proposition \ref{prop:convergence compact} shows that there exists $j_{\varepsilon,2}>0$ such that
$$
\| \nabla M_\beta f_j - \nabla M_\beta f \|_{L^q(3K)} < \varepsilon / 2
$$
for all $j \geq j_{\varepsilon,2}$. Therefore
$$
\| \nabla M_\beta f_j - \nabla M_\beta f \|_{L^q(\R^d)} < \varepsilon
$$
for all $j \geq \max \{j_{\varepsilon,1}, j_{\varepsilon,2} \}$, as desired.




\section{The case $M_\beta^c$ if $d=1$: Proof of Theorem \ref{thm:d=1}}\label{sec:proof thm d=1}

As in the previous section, we first use Proposition \ref{prop:smallness} to show that it suffices to see the convergence inside any compact set $K$. The convergence in the compact set then follows from adapting the ideas for the non-centered $M_\beta$ used by the second author in \cite[Theorem 1]{Madrid2017} or in the proof of Proposition \ref{prop:convergence compactMI}. Note that it is crucial that if $d=1$, $\|f_j-f\|_{W^{1,1}(\R^d)} \to 0$ as $j \to \infty$ ensures uniform convergence.

It is important to note that the monotonicity arguments used in \cite{Madrid2017} to show smallness of $(M_\beta f)'$ outside a compact set do not adapt to the centered maximal operator $M_\beta$ and therefore Proposition \ref{prop:smallness} plays a crucial rôle here.




\appendix

\section{The case $\beta \in (d-1 , d)$}\label{app: (d-1,d)}

The goal of this appendix is to show the limitations of the one dimensional techniques in \cite{Madrid2017}, which only extend to higher dimensions in the limited range $ \beta \in ( d-1, d)$; note that this range is already subsumed by Theorem \ref{thm:beta bigger than 1}.

Let $f \in W^{1,1}$ and $\{f_j\}_{j \in \N} \subset W^{1,1}$ such that $\| f_j - f \|_{W^{1,1}(\R^d)} \to 0$ as $j \to \infty$. By Sobolev embedding and interpolation with $L^1$, one has $\| f_j - f \|_{L^p(\R^d)} \to 0$ for all $1 \leq p \leq \frac{d}{d-1}$ as $j \to \infty$. Note that for $p=d/\beta$ and any ball $B_r$ of radius $r$,
$$
r^\beta \intav_{B_r} |f(y)|\dy \leq  \Big( \int_{\R^d} |f(y)|^p\dy \Big)^{1/p},
$$
so
$$
|M_\beta f_j(x)-M_\beta f(x)| \leq |M_\beta(f_j-f)(x)| \leq \| f_j - f \|_{d/\beta} \to 0 \quad \textrm{as $j \to \infty$}
$$
for all $x \in \R^d$ provided $1 \leq d/\beta \leq \frac{d}{d-1}$, which requires $d-1 \leq \beta < d$. Thus, in this regime of $\beta$, there is uniform convergence of $M_\beta f_j$ to $M_\beta f$. Interpolation with the convergence of $M_\beta f_j$ to $M_\beta f$ in $L^{\frac{d}{d-\beta}, \infty}(\R^d)$, which holds by assumption, yields the convergence on $L^r(\R^d)$ with $\frac{d}{d-\beta} < r \leq  \infty$.

The convergence inside any compact set $K$ of $\nabla M_\beta f_j$ to $\nabla M_\beta f$ on $L^{d/(d-\beta)}$ follows as in the proof of Proposition \ref{prop:convergence compactMI} with some minor modifications. First, it is not needed to remove a small ball near the origin, as in the range $d-1 \leq \beta < d$ there is uniform convergence\footnote{The uniform convergence is a key point in the argument to relate the constants $C_{A,j}$ and $C_A$, and it is not available if $0<\beta<d-1$.} of $M_\beta f_j$ to $M_\beta$ in $\R^d$. Therefore, the compact set can be treated all in one go and can be analysed as the annulus $A(a,b)$. The lack of convergence $\| f_j - f \|_{L^\infty(K)}$ can be overcome replacing the bound \eqref{bound xx} by 
%
%
$$
C_{K}/2 \leq  r_{x,j}^\beta  \intav_{B_{x,j}} |f_j| \leq r_{x,j}^{\alpha} \| f_j \|_{s'} \lesssim r_{x,j}^{\alpha} \| f \|_{s'}
$$
which now holds for $j > \max \{j_1(K), j_2(K)\}$ where $j_2(K)$ is large enough so that $\| f_j - f \|_{s'} \leq \| f \|_{s'}$ for $1 \leq s' \leq \frac{d}{d-1}$, where $\alpha:=\beta - d/s'$. Note that if $\alpha= \beta-d/s'  >0$, one has the uniform lower bound $r_{x,j} \gtrsim (C_K)^{1/\alpha}=: \bar C_K>0$. Thus, it is required that $s'>d/\beta$ and $s'<\frac{d}{d-1}$, which holds if $d-1<\beta<d$.\footnote{The required conditions on $s$ do not allow to obtain the case $\beta=d-1$; in particular, this method does not yield results for the classical case $\beta=0$ if $d=1$.} This immediately yields the desired uniform lower bound on the radius and the convergence in the compact set can be concluded as in Proposition \ref{prop:convergence compactMI}.


In order to show smallness outside a compact set $K$, one can argue as in Proposition \ref{prop:smallness} or, more directly, appeal to the Kinnunen--Saksman inequality \eqref{KS} instead of its refined version in Lemma \ref{lemma:refinement KS}, which is at our disposal in the range $\beta \in (d-1,d)$ for $d>1$, yielding
$$
\int_{K^c} |\nabla M_\beta f_j|^q \lesssim \int_{K^c} |M_{\beta - 1} f|^q + \int_{K^c} |M_{\beta-1} (f-f_j)|^q.
$$
As $f \in L^{\frac{d}{d-1}}(\R^d)$, one has $M_{\beta-1} f \in L^q(\R^d)$ and one can then choose $K$ so that $\int_{K^c} |M_{\beta - 1} f|^q < \varepsilon^q$. For the second term, one can use the boundedness of $M_{\beta -1}$ and the convergence of $f_j$ to $f$ in $L^{\frac{d}{d-1}}$ as $j \to \infty$ to conclude
$$
\int_{K^c} |M_{\beta-1} (f-f_j)|^q \lesssim \Big( \int_{\R^d} |f-f_j|^{\frac{d}{d-1}}\Big)^{q(d-1)/d} \lesssim \varepsilon^q
$$
provided $j$ is large enough.

Finally, it is remarked that the inequality \eqref{KS} does not yield a favourable estimate in one dimension to show smallness outside a compact set. Instead, given a fixed compact set $K=[-R,R]$, the argument in \cite{Madrid2017} for $d=1$ splits 
$\R^d \backslash 3K = Y_1^{j} \cup Y_2^{j}$, where $Y_1^{j}:= \{ x \not \in 3K : |R| \not \in B_{x,j} \}$ and $Y_2^j$ is the complementary set in $\R^d \backslash 3K$. The smallness in $Y_1^j$ is obtained as in Proposition \ref{prop:smallness}. However, to show smallness on $Y_2^j$, the author makes use of the fundamental theorem of calculus after observing some monotonocity properties satisfied $M_\beta f$; this is very attached to the case $d=1$ and does not extend to higher dimensions or the centered case $M_\beta^c$. The more general Proposition \ref{prop:smallness} now subsumes the one dimensional case in \cite{Madrid2017}.




\bibliography{Reference}
\bibliographystyle{amsplain}

\end{document}